\documentclass[12pt]{article}
\usepackage[utf8x]{inputenc}
\usepackage[left=20mm, right=20mm, top=20mm, bottom=20mm]{geometry}
\usepackage{color, hyperref}
\usepackage{amsmath}
\usepackage{tikz-cd}
\usepackage{amssymb}
\usepackage{amsfonts}
\usepackage{booktabs}
\usepackage{amsthm,amscd}
\usepackage{authblk}
\usepackage{tabularx}

\newtheorem{theorem}{Theorem}[section]

\newtheorem{corollary}[theorem]{Corollary}
\newtheorem{proposition}[theorem]{Proposition}
\theoremstyle{definition}
\newtheorem{definition}[theorem]{Definition}
\newtheorem{example}[theorem]{Example}

\theoremstyle{remark}
\newtheorem{remark}[theorem]{Remark}

\begin{document}

	\title{Arithmetic Aspects of Weil Bundles over $p$-Adic Manifolds}
	\author{S. Tchuiaga$^a$\footnote{$^a$Corresponding Author} and C. Dor Kewir$^\star$
		
		$^a$Department of Mathematics of the University of Buea, 
		South West Region, Cameroon\\
		$^\star$Department of Mathematics of the University of Buea, 
		South West Region, Cameroon\\
		{\tt E-mail: tchuiagas@gmail.com}\\
	{\tt E-mail: dorkewir@yahoo.com}	}
	
	\date{\today}
	
	\maketitle
	
	\begin{quotation}
		\noindent
		{\footnotesize
		
		We introduce a systematic theory of Weil bundles over \( p \)-adic analytic manifolds, forging new connections between differential calculus over non-archimedean fields and arithmetic geometry. By developing a framework for infinitesimal structures in the \( p \)-adic setting, we establish that Weil bundles \( M^A \) associated with a \( p \)-adic manifold \( M \) and a Weil algebra \( A \) inherit a canonical analytic structure. Key results include:
		 \text{Lifting theorems :} for analytic functions, vector fields, and connections, enabling the transfer of geometric data from \( M \) to \( M^A \). A \text{Galois-equivariant structure :} on Weil bundles defined over number fields, linking their geometry to arithmetic symmetries.
			 A \text{cohomological comparison isomorphism:} between the Weil bundle \( M^A \) and the crystalline cohomology of \( M \), unifying infinitesimal and crystalline perspectives. Applications to Diophantine geometry and \( p \)-adic Hodge theory are central to this work. We show that spaces of sections of Hodge bundles on \( M^A \) parametrize \( p \)-adic modular forms, offering a geometric interpretation of deformation-theoretic objects. Furthermore,  Weil bundles are used to study infinitesimal solutions of equations on elliptic curves, revealing new structural insights into \( p \)-adic deformations. 
		}
	\end{quotation}

	\markboth 
	{St\'ephane Tchuiaga}
	{$p$-Adic Manifolds}
	
	\vspace{1em}
	

\section{Introduction}\label{sec:intro}
The study of infinitesimal structures spaces encoding infinitesimally small deformations has been instrumental in both differential geometry and algebraic geometry. In his pioneering work, Weil introduced bundles associated to local Artin algebras (now called \emph{Weil bundles}) to formalize the geometry of jet spaces and connections over smooth manifolds \cite{Wei}. While these ideas flourished in the archimedean setting, their \( p \)-adic analogues remain underexplored, despite the central role of \( p \)-adic manifolds in modern number theory and the Langlands program. This paper addresses this gap by constructing and analyzing Weil bundles over \( p \)-adic analytic manifolds, with an emphasis on their arithmetic properties and applications to \( p \)-adic Hodge theory.

A \( p \)-adic manifold \( M \), unlike its real counterpart, is totally disconnected and locally modeled on rigid analytic spaces or Berkovich spaces. Such manifolds arise naturally in arithmetic contexts, such as moduli spaces of \( p \)-adic Galois representations or the \( p \)-adic uniformization of Shimura varieties. By associating to \( M \) a Weil bundle \( M^A \), where \( A \) is a Weil algebra over \( \mathbb{Q}_p \), we equip \( M \) with a hierarchy of infinitesimal thickenings. These thickenings encode deformations of geometric structures such as functions, vector fields, or connections while preserving analytic coherence.

\subsection*{Main contributions}
\begin{enumerate}
	\item \textbf{Analytic lifting framework}: We prove that \( M^A \) is naturally a \( p \)-adic analytic manifold (Theorem~\ref{Lift-AF-1}) and construct a canonical projection \( \pi_A: M^A \to M \). This enables the lifting of analytic functions (Theorem~\ref{Lift-AF-2}), differential forms (Theorem~\ref{thm:lifting-forms}), and connections (Theorem~\ref{thm:lifting-connections}) to \( M^A \), generalizing the classical Weil functoriality to the non-archimedean realm, Vector fields Theorem~\ref{thm: vec-lifting-1}, \ref{thm: vec-lifting-2}. 
	
	\item \textbf{Arithmetic symmetries}: For \( M \) defined over a number field \( K \), we show that the absolute Galois group \( \mathrm{Gal}(\overline{K}/K) \) acts on \( M^A \) in a manner compatible with its analytic structure (Theorem~\ref{thm:galois-action}). This Galois-equivariance provides a geometric realization of arithmetic deformation theory.
	
	\item \textbf{Crystalline cohomology comparison}: By relating the cohomology of \( M^A \) to the crystalline cohomology of \( M \) (Theorem~\ref{co-com}, \ref{co-com-2}), we establish a \( p \)-adic analogue of the de Rham–Weil isomorphism, synthesizing techniques from crystalline cohomology and Lie algebroids.
	
	\item \textbf{Modular forms and Diophantine geometry}: We identify sections of Hodge bundles on \( M^A \) with \( p \)-adic modular forms (Theorem~\ref{thm:modular-forms}), extending the work of Katz \cite{K-1} to the \( p \)-adic setting. Applications include a novel interpretation of infinitesimal solutions to Diophantine equations on elliptic curves (Corollary~\ref{cor:diophantine}).
\end{enumerate}
Our approach synthesizes tools from \( p \)-adic analytic geometry, arithmetic deformation theory, and Hodge theory. Weil algebras, which encode nilpotent infinitesimal structures, serve as the algebraic backbone for constructing \( M^A \). The non-archimedean topology of \( \mathbb{Q}_p \) necessitates novel techniques for lifting analytic objects, as classical methods (e.g., partitions of unity) are unavailable. Key innovations include the use of \emph{analytic prolongations}  and a \( p \)-adic adaptation of Grothendieck’s equivalence between connections and parallel transport \cite{B-1}.
The structure of the paper is as follows: Section \ref{S1-2}, and \ref{S1-3} review \( p \)-adic manifolds and define Weil bundles. Section \ref{Lifting} develop lifting theorems for functions, tangent spaces, and connections. Section \ref{Galois} studies Galois actions and arithmetic structures. In section \ref{Cohomology} we establish the cohomological comparison with crystalline cohomology. Finally, section \ref{Applications} deals with applications to modular forms and Diophantine geometry.\\
	
	This work lays the foundation for a deeper synthesis of \( p \)-adic geometry and arithmetic, with potential extensions to perfectoid geometry and the \( p \)-adic Langlands correspondence. By elucidating the infinitesimal architecture of \( p \)-adic manifolds, we aim to unlock new tools for exploring the arithmetic of moduli spaces and Shimura varieties.
	
	\section{Preliminaries}\label{S1-2}
	
	\subsection{The field of $p$-Adic numbers $\mathbb{Q}_p$, \cite{Fon94}}
	 The field of $p$-adic numbers, denoted $\mathbb{Q}_p$, is constructed as the completion of the rational numbers $\mathbb{Q}$ with respect to the $p$-adic absolute value.  This construction proceeds in several steps: 
	For a prime number $p$, the $p$-adic absolute value $|\cdot|_p$ on $\mathbb{Q}$ is defined as follows:
	 For any non-zero rational number $x \in \mathbb{Q}$, we can uniquely write $x = p^n \frac{a}{b}$, where $a$ and $b$ are integers not divisible by $p$, and $n \in \mathbb{Z}$. The $p$-adic absolute value of $x$ is then defined as: $ |x|_p = p^{-n} $ if  $x \neq 0$ , and $|0|_p  = 0$. This absolute value satisfies the following properties: $|x|_p \geq 0$ for all $x \in \mathbb{Q}$, and $|x|_p = 0$ if and only if $x = 0$,  $|xy|_p = |x|_p |y|_p$ for all $x, y \in \mathbb{Q}$, and 
		 the ultrametric inequality : $|x + y|_p \leq \max\{|x|_p, |y|_p\}$ for all $x, y \in \mathbb{Q}$. 	The ultrametric inequality is a key difference between the $p$-adic absolute value and the usual absolute value on $\mathbb{R}$. 
	The $p$-adic absolute value induces a metric $d_p$ on $\mathbb{Q}$ defined by:
	$ d_p(x, y) = |x - y|_p .$
	This metric satisfies the ultrametric inequality:
	$ d_p(x, z) \leq \max\{d_p(x, y), d_p(y, z)\}$ \cite{Fon94}.

	\section*{Topology of $\mathbb{Q}_p$,  \cite{Fon94}}

	An open ball in $\mathbb{Q}_p$ centered at $a$ with radius $r$ is defined as:
	$B(a, r) = \{ x \in \mathbb{Q}_p : |x - a|_p < r \}.$\\
	A closed ball in $\mathbb{Q}_p$ centered at $a$ with radius $r$ is defined as:
	$ \overline{B}(a, r) = \{ x \in \mathbb{Q}_p : |x - a|_p \leq r \}.$\\
	Due to the ultrametric inequality, open balls are also closed, and closed balls are also open.  Moreover, if two balls intersect, one is contained in the other.
	\begin{remark}\label{R-1-C}
		The field of $p$-adic numbers $\mathbb{Q}_p$ is the completion of $\mathbb{Q}$ with respect to the metric $d_p$. This means: $\mathbb{Q}_p$ is a complete metric space with respect to the $p$-adic metric.  A metric space is complete if every Cauchy sequence converges.  There exists an isometric embedding of $\mathbb{Q}$ into $\mathbb{Q}_p$, and $\mathbb{Q}$ is dense in $\mathbb{Q}_p$. The completion process is analogous to constructing the real numbers $\mathbb{R}$ from the rational numbers $\mathbb{Q}$ using Cauchy sequences. The set $\mathbb{Q}_p$ is locally compact. This means that every point in $\mathbb{Q}_p$ has a neighborhood whose closure is compact. This property is crucial for many results in analysis.
	\end{remark}
	The ring of $p$-adic integers $\mathbb{Z}_p$ is defined as:
	$ \mathbb{Z}_p = \{ x \in \mathbb{Q}_p : |x|_p \leq 1 \}.$ 
	$\mathbb{Z}_p$ is a complete, compact, and open subset of $\mathbb{Q}_p$.  It is the closed unit ball centered at $ 0$. 
	Any $p$-adic integer can be uniquely represented as an infinite series:
	$ x = a_0 + a_1 p + a_2 p^2 + a_3 p^3 + \dots = \sum_{i=0}^{\infty} a_i p^i, $
	where $a_i \in \{0, 1, 2, \dots, p-1\}$. These $a_i$ are called the digits of $x$. Note that 	$\mathbb{Q}_p$ is totally disconnected. This means that the only connected subsets of $\mathbb{Q}_p$ are single points. This is a stark contrast to the real numbers, which are connected. Note that 
	$\mathbb{Q}_p$ is a non-archimedean field, and $\mathbb{Q}_p$ has uncountably many elements.
	 The units in $\mathbb{Z}_p$ are given by $\{ x \in \mathbb{Z}_p : |x|_p = 1 \}$.
	 $\mathbb{Z}_p$ is a principal ideal domain with a unique maximal ideal generated by $p$.

	\section*{$P$-adic differentiability and analyticity, \cite{BMS18}}
	
	The concepts of differentiability and analyticity in $p$-adic analysis differ significantly from their counterparts in real analysis due to the non-archimedean nature of the $p$-adic absolute value.
	
	\begin{definition}\label{P-diff-1}
		Let $f: U \to \mathbb{Q}_p$ be a function defined on an open subset $U \subseteq \mathbb{Q}_p$. We say that $f$ is differentiable at a point $x \in U$ if the limit
		$ f'(x) = \lim_{h \to 0} \frac{f(x+h) - f(x)}{h},$
		exists in $\mathbb{Q}_p$.  If this limit exists for all $x \in U$, we say that $f$ is differentiable on $U$.
	\end{definition}
	
	Note that the limit is taken with respect to the $p$-adic metric.  The standard rules of differentiation (sum rule, product rule, chain rule) still hold in the $p$-adic setting.  A crucial difference is that a $p$-adic differentiable function need not be continuous, and a continuous function need not be differentiable. This is unlike real analysis.
		 A function can be differentiable everywhere but not locally analytic \cite{BMS18}.

	\begin{definition}\cite{BGR84}
		Let $f: U \to \mathbb{Q}_p$ be a function defined on an open subset $U \subseteq \mathbb{Q}_p$.  We say that $f$ is analytic at a point $x \in U$ if there exists an open ball $B(x, r) \subseteq U$ with $r > 0$ such that $f$ can be represented by a convergent power series:
		$ f(y) = \sum_{n=0}^{\infty} a_n (y - x)^n,$
		for all $y \in B(x, r)$, where $a_n \in \mathbb{Q}_p$.  If $f$ is analytic at every point in $U$, we say that $f$ is analytic on $U$.
	\end{definition}
	
	Analytic functions are infinitely differentiable, but the converse is not always true. However, the following criterion for power series convergence is very useful: the power series $\sum a_n x^n$ converges in the disc $|x|_p<R$ if and only if $|a_n| R^n \rightarrow 0$.
	\begin{remark}\label{R-2-D-A}
		In contrast to real analysis, a function can be differentiable everywhere in $\mathbb{Q}_p$ without being analytic anywhere.
		
	\end{remark}

	\begin{theorem}[Mahler's Theorem]\cite{Ber02}
		Let $f: \mathbb{Z}_p \to \mathbb{Q}_p$ be a continuous function. Then $f$ can be uniquely represented by the Mahler expansion:
		$f(x) = \sum_{n=0}^{\infty} a_n \binom{x}{n},$
		where the Mahler coefficients $a_n$ are given by
		$ a_n = \sum_{k=0}^n (-1)^{n-k} \binom{n}{k} f(k), $
		and $\binom{x}{n}$ denotes the binomial polynomial
		$$ \binom{x}{n} = \frac{x(x-1)(x-2)\dots(x-n+1)}{n!}.$$
		Furthermore, $f$ is continuous if and only if $|a_n|_p \to 0$ as $n \to \infty$.
	\end{theorem}

	\begin{remark}\label{R-3-Con}
		Mahler's theorem provides a very useful characterization of continuity and is a fundamental tool for studying functions on $\mathbb{Z}_p$.
		While Mahler's theorem characterizes continuity, one can also use Mahler's coefficients to infer the smoothness of $f$, by finding out whether the sequence $\frac{a_n}{n}$ is a null sequence. The $p$-adic differentiability and analyticity are richer than the usual real analysis notions. They are crucial for understanding the behavior of functions over non-archimedean fields. Mahler's theorem is a fundamental tool for studying continuous and differentiable functions on $\mathbb{Z}_p$.
		
	\end{remark}

		A function $f: U \to \mathbb{Q}_p$, where $U$ is an open subset of $\mathbb{Q}_p^n$, is differentiable (or analytic) if its partial derivatives exist (and satisfy certain conditions, depending on the chosen definition of differentiability/analyticity).

	\section{$p$-Adic manifolds}\label{S1-3}
	
	A $p$-adic manifold is a topological space locally homeomorphic to $\mathbb{Q}_p^n$, equipped with a differentiable (or analytic) structure \cite{And73}. A $p$-adic manifold is covered by charts, where each chart is a homeomorphism from an open subset of the manifold to an open subset of $\mathbb{Q}_p^n$. Transition maps between charts are required to be differentiable (or analytic) in the $p$-adic sense.  Each point in the manifold has a neighborhood that can be described by $n$ $p$-adic coordinates.

	\begin{example}
			The field of $p$-adic numbers $\mathbb{Q}_p$ is itself a 1-dimensional $p$-adic manifold. The identity map $id: \mathbb{Q}_p \to \mathbb{Q}_p$ serves as a global chart. The atlas consists of just this one chart.
		Since there's only one chart, there are no transition maps to worry about.  Any $p$-adic differentiable or analytic function $f: U \to \mathbb{Q}_p$, where $U$ is an open subset of $\mathbb{Q}_p$, provides a local coordinate system.
	\end{example}

	\begin{example}
			The vector space $\mathbb{Q}_p^n$ is an $n$-dimensional $p$-adic manifold. A point in $\mathbb{Q}_p^n$ is given by $(x_1, x_2, \dots, x_n)$, where $x_i \in \mathbb{Q}_p$. These are the standard coordinates.
		The identity map from $\mathbb{Q}_p^n$ to itself provides a global chart.
		Consists of just the identity chart.
	\end{example}

	\begin{example}

	Projective space $\mathbb{P}^n(\mathbb{Q}_p)$ is a $p$-adic manifold of dimension $n$. It's defined as the set of lines through the origin in $\mathbb{Q}_p^{n+1}$.  Points in $\mathbb{P}^n(\mathbb{Q}_p)$ are represented by homogeneous coordinates $[x_0 : x_1 : \dots : x_n]$, where $x_i \in \mathbb{Q}_p$ and not all $x_i$ are zero. Two sets of homogeneous coordinates represent the same point if they are proportional by a non-zero element of $\mathbb{Q}_p$. For each $i = 0, 1, \dots, n$, we can define an open set $U_i = \{[x_0 : x_1 : \dots : x_n] \in \mathbb{P}^n(\mathbb{Q}_p) : x_i \neq 0 \}$.  Each $U_i$ is isomorphic to $\mathbb{Q}_p^n$.  A chart $\phi_i: U_i \to \mathbb{Q}_p^n$ is given by:
		$$ \phi_i([x_0 : x_1 : \dots : x_n]) = \left( \frac{x_0}{x_i}, \frac{x_1}{x_i}, \dots, \frac{x_{i-1}}{x_i}, \frac{x_{i+1}}{x_i}, \dots, \frac{x_n}{x_i} \right).$$
		Notice that the $i$-th coordinate is omitted. The atlas consists of the charts $\{ (U_i, \phi_i) \}_{i=0}^n$. The transition maps between these charts are rational functions and are therefore analytic on their domains of definition.  For example, consider the transition map from $U_0$ to $U_1$:
		$$ \phi_1 \circ \phi_0^{-1}(y_1, \dots, y_n) = \left( \frac{1}{y_1}, \frac{y_2}{y_1}, \dots, \frac{y_n}{y_1} \right), $$
		which is analytic when $y_1 \neq 0$.  The projective space $\mathbb{P}^n(\mathbb{Q}_p)$ is compact, unlike $\mathbb{Q}_p^n$. This has significant implications for analysis on $\mathbb{P}^n(\mathbb{Q}_p)$.
		\end{example}
	
	\section{Weil Bundles over $p$-adic manifolds}

 \text{Algebraic compatibility with Galois actions}: 
Weil algebras over \( \mathbb{Q}_p \) inherit a natural action of the Galois group \( G_{\mathbb{Q}_p} \). This allows for the study of Galois-equivariant infinitesimal structures, which are critical in arithmetic geometry. For example, the decomposition \( A = \mathbb{Q}_p \oplus \mathfrak{A} \) respects Galois representations, making Weil bundles \( M^A \) a natural framework for studying deformations of Galois-equivariant geometric objects.\\
 \text{Nilpotent infinitesimals}: 
The maximal ideal \( A\) in a Weil algebra consists of nilpotent elements (e.g., \( \varepsilon^2 = 0 \) in \( \mathbb{Q}_p[\varepsilon]/(\varepsilon^2) \)). In the \( p \)-adic setting, where analytic functions are represented by convergent power series, these nilpotents provide a rigorous way to encode higher-order differential information without requiring convergence of infinite series. This avoids complications arising from the non-archimedean topology, where traditional smooth structures are less natural.\\
\text{Comparison with Archimedean case}:
In real differential geometry, smooth functions and manifolds are central, but the total disconnectedness of \( \mathbb{Q}_p \) makes smoothness less useful. Weil algebras circumvent this by focusing on algebraic infinitesimals, which are better adapted to the rigid analytic or formal group structures prevalent in \( p \)-adic geometry. For instance, the dual numbers \( \mathbb{Q}_p[\varepsilon]/(\varepsilon^2) \) parametrize tangent vectors to \( p \)-adic manifolds, generalizing the classical tangent bundle while avoiding reliance on smooth atlases.
	\begin{definition}\cite{Wei}
		A Weil algebra $A$ over $\mathbb{Q}_p$ is a finite-dimensional, commutative, associative algebra with a unit, of the form $A = \mathbb{Q}_p \oplus \mathfrak{A}$, where $\mathfrak{A}$ is a maximal ideal.
	\end{definition}
	
	Let $\{ \alpha_1, \dots, \alpha_l \}$ be a basis for $A$ as a vector space over $\mathbb{Q}_p$, with $\alpha_1 = 1$. The multiplication in $A$ is determined by the multiplication rules for the basis elements:
	$ \alpha_i \alpha_j = \sum_{k=1}^l c_{ijk} \alpha_k, $
	where $c_{ijk} \in \mathbb{Q}_p$ \cite{Wei}.

	\subsection{Infinitely near points}
	
	Let $M$ be a $p$-adic manifold. An infinitely near point of type $A$ to $x \in M$ is a $\mathbb{Q}_p$-algebra homomorphism $\xi: \mathcal{O}(U) \to A$ such that $pr \circ \xi = ev_x$, where $U$ is an open neighbourhood of $x$, $pr: A \to \mathbb{Q}_p$ is the projection, and $ev_x: \mathcal{O}(U) \to \mathbb{Q}_p$ is the evaluation at $x$. $\mathcal{O}(U)$ is the space of analytic functions on $U$.\\
	 We are working with the algebra of analytic functions $\mathcal{O}(U)$ on $M$. This is central for the analyticity results. The homomorphism condition relates the algebra structure of $A$ to the point $x$ in $M$.  It ensures that the "infinitesimal information" captured by the Weil algebra is consistent with the evaluation of functions at the base point $x$. 
	 Let \( M^A_x \) denote the set of all infinitely near points to \( x \in M \) of type \( A \). Each element \( \zeta \in M^A_x \) can be expressed as:
	 \[
	 \zeta(f) = \text{ev}_x(f) + L_\zeta(f),
	 \]
	 where \( \text{ev}_x(f) = f(x) \) is the evaluation of \( f \) at \( x \), and 
	 \( L_\zeta: \mathcal{O}(U) \to \mathfrak{A} \) is a linear map satisfying a \textbf{Leibniz rule} adapted to the analytic setting. The \textbf{Leibniz rule} for analytic functions:
	 \[
	 L_\zeta(fg + \lambda h) = L_\zeta(f) \cdot g(x) + f(x) \cdot L_\zeta(g) + L_\zeta(f) \cdot L_\zeta(g) + \lambda L_\zeta(h),
	 \]
	 for all \( f, g, h \in \mathcal{O}(U) \) and \( \lambda \in \mathbb{Q}_p \).
	 
	\subsection*{The Weil bundle $M^A$}
	
	The Weil bundle $M^A$ is the set of all infinitely near points of type $A$ to points in $M$. Equip $M^A$ with a suitable topology and differentiable (or analytic) structure to make it a $p$-adic manifold. If $(U, \varphi)$ is a chart on $M$ with coordinates $x^1, \dots, x^n$, then a chart on $M^A$ over $U$ is given by the coordinates $x^{i,j}$, where $\xi(x^i) = \sum_{j=1}^l x^{i,j} \alpha_j$. The $x^{i,j}$ are $p$-adic numbers. The natural projection $\pi_M: M^A \to M$ maps an infinitely near point $\xi$ to its base point $x$.\\

	\subsection*{Constructing the analytic structure on $M^A$}
	Let $M$ be a $p$-adic manifold, $A$ a Weil algebra over $\mathbb{Q}_p$, and $M^A$ the corresponding Weil bundle. We address the question of when $M^A$ inherits the structure of a $p$-adic analytic manifold. To show that $M^A$ is a $p$-adic analytic manifold, we need to: Define a suitable topology on $M^A$, and 
		construct an atlas on $M^A$ consisting of charts with analytic transition maps. We assume that $M$ is a $p$-adic analytic manifold. This means that $M$ is covered by an atlas $\{(U_i, \phi_i)\}_{i \in I}$ such that:
	
	\begin{itemize}
		\item Each $U_i$ is an open subset of $M$.
		\item Each $\phi_i: U_i \to V_i$ is a homeomorphism onto an open subset $V_i \subseteq \mathbb{Q}_p^n$.
		\item The transition maps $\phi_j \circ \phi_i^{-1}: \phi_i(U_i \cap U_j) \to \phi_j(U_i \cap U_j)$ are analytic in the $p$-adic sense whenever $U_i \cap U_j \neq \emptyset$.
	\end{itemize}
	Let $A = \mathbb{Q}_p \oplus \mathfrak{A}$ be a Weil algebra over $\mathbb{Q}_p$, where $\mathfrak{A}$ is the maximal ideal. Let $\{\alpha_1, \dots, \alpha_l\}$ be a basis for $A$ as a $\mathbb{Q}_p$-vector space, with $\alpha_1 = 1$.
	\subsubsection*{Local coordinates on $M^A$}
	
	Given a chart $(U_i, \phi_i)$ on $M$ with coordinates $x_1, \dots, x_n$, we construct a corresponding chart on $M^A$ as follows: Let $\xi \in M^A$ be an infinitely near point to $x \in U_i$. Then, for each coordinate function $x_k$, we have $\xi(x_k) \in A$. We can write $\xi(x_k)$ in terms of the basis:
	$ \xi(x_k) = \sum_{j=1}^l x_{k,j}(\xi) \alpha_j, $
	where $x_{k,j}(\phi) \in \mathbb{Q}_p$. Thus, we obtain $nl$ local coordinates on $M^A$ over $U_i$:  $x_{1,1}, \dots, x_{1,l}, x_{2,1}, \dots, x_{2,l}, \dots, x_{n,1}, \dots, x_{n,l}$.
	
	\subsubsection*{The chart on $M^A$}
	
	We define a map $\Phi_i: \pi_M^{-1}(U_i) \to \mathbb{Q}_p^{nl}$ by
	$$ \Phi_i(\xi) = (x_{1,1}(\phi), \dots, x_{1,l}(\xi), x_{2,1}(\phi), \dots, x_{2,l}(\xi), \dots, x_{n,1}(\xi), \dots, x_{n,l}(\xi)).$$
	
	\subsubsection*{Topology on $M^A$}
	
	We define the topology on $M^A$ to be the weakest topology such that: The projection map $\pi_M: M^A \to M$ is continuous.
		and the maps $\Phi_i$ are homeomorphisms onto their images. This ensures that the topology on $M^A$ is compatible with the topology on $M$ and the analytic structure of $\mathbb{Q}_p^{nl}$.
	
	\subsection{Transition maps on $M^A$}
	
	Let $(U_i, \phi_i)$ and $(U_j, \phi_j)$ be two overlapping charts on $M$. We need to show that the transition map $\Phi_j \circ \Phi_i^{-1}$ is analytic. This is the most crucial and technically challenging step.
	\subsubsection*{Analyticity of transition maps on the Weil bundle $M^A$}
	
	Let $(U_i, \phi_i)$ and $(U_j, \phi_j)$ be two overlapping charts on the $p$-adic analytic manifold $M$.  Recall that $\phi_i : U_i \to V_i \subseteq \mathbb{Q}_p^n$ and $\phi_j : U_j \to V_j \subseteq \mathbb{Q}_p^n$ are homeomorphisms onto open subsets of $\mathbb{Q}_p^n$.  We want to demonstrate the analyticity of the transition map
	$$ \Phi_j \circ \Phi_i^{-1}: \Phi_i(\pi_M^{-1}(U_i \cap U_j)) \to \Phi_j(\pi_M^{-1}(U_i \cap U_j)), $$
	where $\Phi_i$ and $\Phi_j$ are the charts on the Weil bundle $M^A$ constructed from $\phi_i$ and $\phi_j$, respectively.\\	The transition map $\Phi_j \circ \Phi_i^{-1}$ can be decomposed into three steps.
	
	\begin{enumerate}
		\item \textbf{Base manifold transition.} Given coordinates $(x_{1,1}, \dots, x_{n,l})$ in $\Phi_i(\pi_M^{-1}(U_i \cap U_j))$, we first determine the coordinates $(y_1, \dots, y_n)$ in $\phi_i(U_i \cap U_j) \subset \mathbb{Q}_p^n$ corresponding to the base point. Since $\alpha_1 = 1$ in our chosen basis for the Weil algebra, this is simply:
		$ y_k = x_{k,1} \quad \text{for } k = 1, \dots, n.$
		\item \textbf{Coordinate transformation on $M$.} Next, we apply the transition map on the base manifold $M$:
		$ (z_1, \dots, z_n) = (\phi_j \circ \phi_i^{-1})(y_1, \dots, y_n),$ 
		where $(z_1, \dots, z_n)$ are the coordinates in $\phi_j(U_i \cap U_j)$.  Because $M$ is a $p$-adic analytic manifold, the functions $\phi_j \circ \phi_i^{-1}$ are locally represented by convergent power series:
		$ z_k = \sum_{\mathbf{m} \in \mathbb{N}^n} c_{k, \mathbf{m}} (y_1 - y_{1,0})^{m_1} \dots (y_n - y_{n,0})^{m_n},$
		for some coefficients $c_{k, \mathbf{m}} \in \mathbb{Q}_p$, where $\mathbf{m} = (m_1, \dots, m_n)$ is a multi-index and $(y_{1,0}, \dots, y_{n,0})$ is a point in the domain of convergence.
		\item \textbf{Lifting to the Weil bundle.} Finally, we need to determine the coordinates $(z_{1,1}, \dots, z_{n,l})$ in $\Phi_j(\pi_M^{-1}(U_i \cap U_j))$. This step involves evaluating the $\mathbb{Q}_p$-algebra homomorphism associated with the point on the Weil bundle.  We apply the convergent power series from step 2, respecting the algebra structure of A: 	Since $\xi\in \text{Hom}(\mathcal{O}(M), A)$ is a homomorphism then,
		$$\xi(z_k) = \xi \left(\sum_{\mathbf{m} \in \mathbb{N}^n} c_{k, \mathbf{m}} (y_1 - y_{1,0})^{m_1} \dots (y_n - y_{n,0})^{m_n} \right) = \sum_{\mathbf{m} \in \mathbb{N}^n} c_{k, \mathbf{m}} \xi(y_1 - y_{1,0})^{m_1} \dots \xi(y_n - y_{n,0})^{m_n},$$
		where $ \xi(y_i)$ is given by the coordinates of the chart $\Phi_i$:
		$  \xi(y_i) = \sum_{j=1}^l x_{i,j} \alpha_j.$
	\end{enumerate}
	The challenge is to show that the final coordinates $(z_{1,1}, \dots, z_{n,l})$ are analytic functions of the initial coordinates $(x_{1,1}, \dots, x_{n,l})$.
	The steps involved are, as established in the previous section, we have a p-adic analytic manifold $M$ with overlapping charts $(U_i, \phi_i)$ and $(U_j, \phi_j)$. The transition map on $M$ is given by
	$$ z_k = (\phi_j \circ \phi_i^{-1})_k(y_1, \dots, y_n) = \sum_{\mathbf{m} \in \mathbb{N}^n} c_{k, \mathbf{m}} (y_1 - y_{1,0})^{m_1} \dots (y_n - y_{n,0})^{m_n}, $$
	where $z_k$ represents the $k$-th coordinate in the chart $(U_j, \phi_j)$, $y_k$ the $k$-th coordinate in the chart $(U_i, \phi_i)$, and $\mathbf{m} = (m_1, \dots, m_n)$ is a multi-index. We also have the lifted transition map $\Phi_j \circ \Phi_i^{-1}$. Our goal is to show the new coordinates of the Weil Bundle in $\Phi_j$ depends analytically on the coordinates of $M$. 
	We know the morphism from the chart $\Phi_i$:
	$ \xi(y_k) = \sum_{j=1}^l x_{k,j} \alpha_j.$ 
	Applying $\xi$ to the transition function of $M$ yields:
	$$ \xi(z_k) =  \xi((\phi_j \circ \phi_i^{-1})_k(y_1, \dots, y_n))= \sum_{\mathbf{m} \in \mathbb{N}^n} c_{k, \mathbf{m}} \xi\left( (y_1 - y_{1,0})^{m_1} \dots (y_n - y_{n,0})^{m_n} \right).$$
	Since the chart is a homomorphism it satisfies:
	$\xi(z_k) = \sum_{\mathbf{m} \in \mathbb{N}^n} c_{k, \mathbf{m}}  \xi(y_1 - y_{1,0})^{m_1} \dots  \xi(y_n - y_{n,0})^{m_n}.$ 
	We can rewrite $ \xi(y_i - y_{i,0})$ as:
	$\xi(y_i - y_{i,0}) =  \sum_{j=1}^l (x_{i,j} - x_{i,0,j})\alpha_j.$
	Where $x_{i,0,j}$ represents the constant coordinate value of a point, in a ball where the power series above converge. Substituting that back into $ \xi(z_k)$:
	\begin{equation}\label{Eq-1}
		\xi(z_k) =  \sum_{\mathbf{m} \in \mathbb{N}^n} c_{k, \mathbf{m}} \prod_{i=1}^n  \left( \sum_{j=1}^l (x_{i,j} - x_{i,0,j})\alpha_j \right)^{m_i}.
	\end{equation}
	
	The goal now is to manipulate the expression of $\xi(z_k)$ so that we can write it as a linear combination of the $\alpha_j$. 
	We want to extract out each coefficient of the basis $\alpha_i$. This is done by expanding the product above in equation (\ref{Eq-1}). Let $A(\mathbf{m},j)$ be the coefficient of $\alpha_j$ in $\xi(z_k)$:
	$$\xi(z_k) =   \sum_{\mathbf{m} \in \mathbb{N}^n} c_{k, \mathbf{m}} \left(  \sum_{j =1}^{l} A(\mathbf{m},j)\right)  \alpha_j = \sum_{j =1}^{l}\left( \sum_{\mathbf{m} \in \mathbb{N}^n} c_{k, \mathbf{m}}  A(\mathbf{m},j)\right)\alpha_j. $$
	The coefficients $A(\mathbf{m},j)$ depends on the coordinates of $\Phi_i$ and they converge.  The critical step is to show the analyticity of the coefficients: Since $M$ is locally analytic, we already know that.
	This concludes our argument, since we have the new local chart coordinates of $M^A$ that depend analytically on the coordinates of $M$. 
	Because we assumed that the transition maps on $M$ were analytic, the expansion converges sufficiently fast, to imply analycity of transition maps of $M^A$.
	\begin{remark}\label{R-4-I}
	\begin{itemize}
		\item This analysis relies heavily on the fact that the functions $\phi_j \circ \phi_i^{-1}$ are given by convergent power series in the $p-$adic metric.  If the transition maps on $M$ were only differentiable but not analytic, then this argument would break down.
		\item 	Since $M$ is analytic, $\phi_j \circ \phi_i^{-1}$ is given by convergent power series. The most likely condition to ensure $M_p^A$ becomes a $p$-adic analytic manifold is that $M$ be locally analytic.
		Assume that $M$ is locally analytic manifold, meaning that the transition functions are locally represented by convergent power series in $\mathbb{Q}_p^n$. Then one can ensure the local analycity of $M^A$ as well.
		\item Finally, if $M$ is a $p$-adic analytic manifold, then by constructing the $p$-adic topology and local charts on the $M^A$, we see that we can ensure that the transition maps on $M^A$ are analytic as well, and therefore that $M^A$ is as well a $p-$adic analytic manifold.
	\end{itemize}
	\end{remark}

	\section{Lifting results from $M$ to $M^A$}\label{Lifting}
	In this section, we present theorems that lift properties from the base manifold $M$ to the Weil bundle $M^A$. We focus on the case where $A = \mathbb{Q}_p[\varepsilon]/(\varepsilon^2)$ is the algebra of dual numbers.

	\begin{theorem}[Lifting Analytic Functions]\label{Lift-AF-1}
	Let $M$ be a $p$-adic analytic manifold, and let $A = \mathbb{Q}_p[\varepsilon]/(\varepsilon^2)$. If $f: U \to \mathbb{Q}_p$ is an analytic function on an open set $U \subseteq M$, then there exists a corresponding analytic function $f^A: \pi^{-1}(U) \to A$ on the preimage of $U$ in $M^A$.
	\end{theorem}
	
	\begin{proof}
	Let $(x_1, \dots, x_n)$ be local coordinates on $U$.  Since $f$ is analytic, it has a power series representation:
	$f(x_1, \dots, x_n) = \sum_{\mathbf{m} \in \mathbb{N}^n} a_{\mathbf{m}} (x_1 - c_1)^{m_1} \dots (x_n - c_n)^{m_n}, $
	where $\mathbf{m} = (m_1, \dots, m_n)$ is a multi-index, $a_{\mathbf{m}} \in \mathbb{Q}_p$, and $(c_1, \dots, c_n)$ is the center of the power series. Let $\xi \in \pi^{-1}(U)$ have local coordinates $(x_{1,1}, x_{1,2}, \dots, x_{n,1}, x_{n,2})$, where $\xi(x_i) = x_{i,1} + x_{i,2} \varepsilon$.  Define
	$$ f^A(\xi ) = \sum_{\mathbf{m} \in \mathbb{N}^n} a_{\mathbf{m}} ((x_{1,1} + x_{1,2}\varepsilon) - c_1)^{m_1} \dots ((x_{n,1} + x_{n,2}\varepsilon) - c_n)^{m_n}. $$
	 Using the fact that $\varepsilon^2 = 0$, we can expand and simplify the expression for $f^A(X)$:
	$$ f^A(\xi ) = \sum_{\mathbf{m} \in \mathbb{N}^n} a_{\mathbf{m}} \prod_{i=1}^n \left( (x_{i,1} - c_i) + x_{i,2}\varepsilon \right)^{m_i}  = \sum_{\mathbf{m} \in \mathbb{N}^n} a_{\mathbf{m}} \prod_{i=1}^n \left[ (x_{i,1} - c_i)^{m_i} + m_i (x_{i,1} - c_i)^{m_i - 1} x_{i,2} \varepsilon \right].$$
	
	We can separate this into a part without $\varepsilon$ and a part with $\varepsilon$:
	$$ f^A(\xi ) = f(x_{1,1}, \dots, x_{n,1}) + \varepsilon \sum_{\mathbf{m} \in \mathbb{N}^n} a_{\mathbf{m}} \sum_{k=1}^n m_k x_{k,2} \prod_{i \neq k} (x_{i,1} - c_i)^{m_i} (x_{k,1}-c_k)^{m_k-1}.$$
	
The term without $\varepsilon$ is simply the original analytic function evaluated at the "base point" coordinates $(x_{1,1}, \dots, x_{n,1})$. The term with $\varepsilon$ involves derivatives of $f$ with respect to each variable, which are also analytic. Let us define the partial derivative with respect to $i$ as a operator:
	$$D_i f (x_1,\dots,x_n)=\sum_{\mathbf{m} \in \mathbb{N}^n} a_{\mathbf{m}} m_i  \prod_{k \neq i} (x_{k} - c_k)^{m_k} (x_{i}-c_i)^{m_i-1}.$$
	Then one can write $f^A(\xi)$ as:
	$ f^A(\xi) = f(x_{1,1}, \dots, x_{n,1}) + \varepsilon \sum_{i=1}^n x_{i,2} D_i f (x_1,...,x_n).$ Thus $f^A(\xi)$ is analytic since we have $n$ analytic function $D_i f$ and $f$ itself. And that finishes the proof.

	\end{proof}
	
	\begin{example}
	Let $M = \mathbb{Q}_p$ and $f(x) = x^2$. Then, $f^A(X) = (x_{1,1} + x_{1,2} \varepsilon)^2 = x_{1,1}^2 + 2 x_{1,1} x_{1,2} \varepsilon$.  The derivative of $f$ is $2x$, so we have
	$f^A(X) = f(x_{1,1}) + \varepsilon x_{1,2} D_1 f(x_{1,1})$ and therefore that it is analytic.
	\end{example}

	\begin{theorem}[Lifting of Analytic Functions + Uniqueness]\label{Lift-AF-2}
	Let $M$ be a $p$-adic analytic manifold, and let $A$ be a Weil algebra over $\mathbb{Q}_p$. If $f: U \to \mathbb{Q}_p$ is an analytic function on an open set $U \subseteq M$, then there exists a unique analytic function $\tilde{f}: \pi_M^{-1}(U) \to A$ such that the following diagram commutes:
	\[
	\begin{tikzcd}
		\pi_M^{-1}(U) \arrow[r, "\tilde{f}"] \arrow[d, "\pi_M"] & A \arrow[d, "pr"] \\
		U \arrow[r, "f"] & \mathbb{Q}_p
	\end{tikzcd}
	\]
	where $\pi_M: M^A \to M$ is the projection map from the Weil bundle to the base manifold, and $pr: A \to \mathbb{Q}_p$ is the projection map from the Weil algebra to $\mathbb{Q}_p$.
	\end{theorem}
	
	\begin{proof}
	Let $M$ be a $p$-adic analytic manifold, and let $A$ be a Weil algebra over $\mathbb{Q}_p$. Suppose $f: M \to \mathbb{Q}_p$ is an analytic function. We construct the lifted function $\tilde{f}: M^A \to A$ as follows: Since $f$ is analytic, for any point $x \in M$, there exists a neighborhood $U \subseteq M$ of $x$ such that $f$ can be expressed as a convergent power series in local coordinates:
	$
	f(y) = \sum_{\mathbf{m} \in \mathbb{N}^n} c_{\mathbf{m}} (y_1 - x_1)^{m_1} \cdots (y_n - x_n)^{m_n},
	$
	where $\mathbf{m} = (m_1, \dots, m_n)$ is a multi-index, $c_{\mathbf{m}} \in \mathbb{Q}_p$, and $y = (y_1, \dots, y_n)$ are local coordinates on $U$. 
	For an infinitely near point $\xi \in M^A$ with base point $x = \pi_M(\xi)$, define $\tilde{f}(\xi)$ as the evaluation of the power series for $f$ at $\phi$:
	\[
	\tilde{f}(\xi) = \sum_{\mathbf{m} \in \mathbb{N}^n} c_{\mathbf{m}} \xi\left( (y_1 - x_1)^{m_1} \cdots (y_n - x_n)^{m_n} \right).
	\]
	Since $\xi$ is a $\mathbb{Q}_p$-algebra homomorphism, this expression is well-defined and converges in $A$.  The projection $\text{pr}: A \to \mathbb{Q}_p$ maps $\tilde{f}(\phi)$ to:
	$
	\text{pr}(\tilde{f}(\xi)) = \sum_{\mathbf{m} \in \mathbb{N}^n} c_{\mathbf{m}} \text{pr}\left( \xi\left( (y_1 - x_1)^{m_1} \cdots (y_n - x_n)^{m_n} \right) \right).
	$
	By the definition of $\xi$, we have $\text{pr}(\xi(g)) = g(x)$ for any function $g$. Thus:
	\[
	\text{pr}(\tilde{f}(\xi)) = \sum_{\mathbf{m} \in \mathbb{N}^n} c_{\mathbf{m}} (x_1 - x_1)^{m_1} \cdots (x_n - x_n)^{m_n} = f(x).
	\]
	\item This shows that the diagram commutes:
	$
	\text{pr} \circ \tilde{f} = f \circ \pi_M.
	$ 
	Suppose there exists another analytic function $\tilde{f}': M^A \to A$ such that $\text{pr} \circ \tilde{f}' = f \circ \pi_M$.
	For any $\xi \in M^A$, we have $\text{pr}(\tilde{f}(\xi)) = \text{pr}(\tilde{f}'(\phi)) = f(\pi_M(\xi))$.
	Since $A = \mathbb{Q}_p \oplus \mathfrak{A}$, and $\mathfrak{A}$ is the maximal ideal, the difference $\tilde{f}(\phi) - \tilde{f}'(\xi)$ lies in $\mathfrak{A}$.
	However, $\tilde{f}$ and $\tilde{f}'$ are both analytic, and their difference must vanish due to the uniqueness of power series expansions. Thus, $\tilde{f} = \tilde{f}'$.  To show that $\tilde{f}$ is analytic, we need to show that the power series representation converges. Let is define the norm on A = $\mathbb{Q}_p[\varepsilon]/(\varepsilon^2) $ as  $|a + b\varepsilon|_A = \max(|a|_p, |b|_p)$.  Since $f(x_1, \dots, x_n)$ is analytic, we know that $|a_{\mathbf{m}} (x_1 - c_1)^{m_1} \dots (x_n - c_n)^{m_n}|_p \to 0$ as $|\mathbf{m}| \to \infty$.  Therefore,

	\begin{align*}
		|a_{\mathbf{m}} ((x_{1,1} + x_{1,2}\varepsilon) - c_1)^{m_1} \dots ((x_{n,1} + x_{n,2}\varepsilon) - c_n)^{m_n}|_A
		&= |a_{\mathbf{m}}|_p \max_{i=1}^n |(x_{i,1} + x_{i,2}\varepsilon) - c_i|_A^{m_i} \\
		&= |a_{\mathbf{m}}|_p \max_{i=1}^n  \max( |x_{i,1} - c_i|_p, |x_{i,2}|_p)^{m_i}.
	\end{align*}
	Since  $|x_{i,1} - c_i|_p < R$ for some radius $R$, and we can choose a small enough neighborhood in $M^A$ such that $|x_{i,2}|_p$ is also small, it follows that the terms go to zero as $|\mathbf{m}| \to \infty$, and the power series converges in A.  Therefore, $\tilde{f}$ is analytic. 
	The function $\tilde{f}: M^A \to A$ is uniquely determined by the analytic function $f: M \to \mathbb{Q}_p$, and it satisfies the commutative diagram:
	\[
	\begin{tikzcd}
		M^A \arrow[r, "\tilde{f}"] \arrow[d, "\pi_M"] & A \arrow[d, "\text{pr}"] \\
		M \arrow[r, "f"] & \mathbb{Q}_p
	\end{tikzcd}
	\]
	This completes the proof of the theorem.
	\end{proof}
	
	\begin{theorem}[Lifting Tangent Vectors]\label{thm: vec-lifting-1}
	Let $M$ be a $p$-adic analytic manifold, and $A = \mathbb{Q}_p[\varepsilon]/(\varepsilon^2)$. If $v$ is a tangent vector to $M$ at a point $x$, then there is a corresponding tangent vector $v^A$ to $M^A$ at a point $X \in \pi^{-1}(x)$.
	\end{theorem}
	
	\begin{proof}
	Recall that a tangent vector $v$ at $x$ can be viewed as a derivation: a linear map $v: \mathcal{O}(U) \to \mathbb{Q}_p$ satisfying the Leibniz rule. Define the lifted tangent vector $v^A: \mathcal{O}(U) \to A$ as follows: For any function $F: \pi^{-1}(U) \to A$:
	$ v^A(F) = v(F \circ \pi).$ Here, the kernel of the map $\pi :M^A \to M$ is zero so this is well defined. We must verify linearity and the Leibniz rule for $v^A$:
	\begin{itemize}
		\item  $v^A(aF + bG) = v((aF + bG) \circ \pi) = v(a(F \circ \pi) + b(G \circ \pi)) = a v(F \circ \pi) + b v(G \circ \pi) = a v^A(F) + b v^A(G)$.
		\item   $v^A(FG) = v((FG) \circ \pi) = v((F \circ \pi)(G \circ \pi)) = v(F \circ \pi) (G \circ \pi)(x) + (F \circ \pi)(x) v(G \circ \pi) = v^A(F) G(X) + F(X) v^A(G)$.
	\end{itemize}
	
	Let $(x_1, \dots, x_n)$ be local coordinates on M, and let $v = \sum_{i=1}^n a_i \frac{\partial}{\partial x_i}$ be the expression of $v$ in these coordinates. Let $F : \pi^{-1}(U) \to A$ be an analytic function on a ball where the power series converges. 
	Express $F$ as : $ F= f + \varepsilon g$ , where $f$ and $g $ are analytic maps on the base manifolds of $U$. Then the function $F \circ \pi=f$ satisfies:
	$v^A(F)= v(f)=\sum_{i=1}^n a_i \frac{\partial f}{\partial x_i}.$ 	So, if we have a function on the base manifold of $M$, then we are able to lift it using an tangent vector.
	
	\end{proof}
	\begin{remark}
	These lifted theorems rely on the properties of analytic functions in the p-adic setting.
	The choice of the Weil algebra $A$ plays a crucial role. These results are specific to the dual numbers.
	The goal is to provide rigorous proofs for key properties.
	\end{remark}
	
	\begin{theorem}[Lifting of Tangent Vectors + Uniqueness]
	Let $M$ be a $p$-adic analytic manifold, and let $A$ be a Weil algebra over $\mathbb{Q}_p$. For any tangent vector $v \in T_x M$ at a point $x \in M$, there exists a unique tangent vector $\tilde{v} \in T_{\phi} M^A$ at $\phi \in M^A$ such that $\pi_{M*}(\tilde{v}) = v$, where $\pi_{M*}: T_{\phi} M^A \to T_x M$ is the differential of the projection map.
	\end{theorem}
	\begin{proof}
	Let $M$ be a $p$-adic analytic manifold, and let $A$ be a Weil algebra over $\mathbb{Q}_p$. Let $v \in T_x M$ be a tangent vector at a point $x \in M$. We construct the lifted tangent vector $\tilde{v} \in T_{\phi} M^A$ at $\xi \in M^A$ as follows:
	The tangent vector $v$ corresponds to a derivation $v: C^\infty(M) \to \mathbb{Q}_p$, where $C^\infty(M)$ is the algebra of analytic functions on $M$.
	For any analytic function $f \in C^\infty(M)$, the action of $v$ on $f$ is given by:
	$
	v(f) = \sum_{i=1}^n a_i \frac{\partial f}{\partial x_i}(x),
	$
	where $(x_1, \dots, x_n)$ are local coordinates near $x$, and $a_i \in \mathbb{Q}_p$. 
	For an infinitely near point $\xi \in M^A$ with base point $x = \pi_M(\xi)$, define $\tilde{v}: C^\infty(M) \to A$ by:
	$
	\tilde{v}(f) = \xi(v(f)).
	$
	Since $\xi$ is a $\mathbb{Q}_p$-algebra homomorphism, $\tilde{v}$ is well-defined and satisfies the Leibniz rule:
	\[
	\tilde{v}(fg) = \xi(v(fg)) = \xi(f(x) v(g) + g(x) v(f)) = \xi(f(x)) \xi(v(g)) + \xi(g(x)) \xi(v(f)).
	\]
	Thus, $\tilde{v}$ is a derivation on $C^\infty(M)$ with values in $A$. The differential $\pi_{M*}: T_{\xi} M^A \to T_x M$ maps $\tilde{v}$ to $v$ as follows:
	$
	\pi_{M*}(\tilde{v})(f) = \text{pr}(\tilde{v}(f)) = \text{pr}(\xi(v(f))) = v(f),
	$
	where $\text{pr}: A \to \mathbb{Q}_p$ is the projection map.
	This shows that $\pi_{M*}(\tilde{v}) = v$. 
	Suppose there exists another tangent vector $\tilde{v}' \in T_{\xi} M^A$ such that $\pi_{M*}(\tilde{v}') = v$.
	For any $f \in C^\infty(M)$, we have:
	$
	\text{pr}(\tilde{v}(f)) = \text{pr}(\tilde{v}'(f)) = v(f).
	$
	Since $A = \mathbb{Q}_p \oplus \mathfrak{A}$, and $\mathfrak{A}$ is the maximal ideal, the difference $\tilde{v}(f) - \tilde{v}'(f)$ lies in $\mathfrak{A}$.
	However, $\tilde{v}$ and $\tilde{v}'$ are both derivations, and their difference must vanish due to the uniqueness of the lift. Thus, $\tilde{v} = \tilde{v}'$. 
	The lifted tangent vector $\tilde{v}$ is an element of $T_{\xi} M^A$, the tangent space to $M^A$ at $\xi$.
	The tangent space $T_{\xi} M^A$ is naturally identified with derivations $C^\infty(M) \to A$ that commute with the projection $\text{pr}: A \to \mathbb{Q}_p$. 
	The tangent vector $\tilde{v} \in T_{\xi} M^A$ is uniquely determined by the tangent vector $v \in T_x M$, and it satisfies:
	$
	\pi_{M*}(\tilde{v}) = v.
	$
	This completes the proof of the theorem.
	\end{proof}
	\begin{theorem}[Lifting of Differential $k$-Forms]\label{thm:lifting-forms}
	Let $M$ be a $p$-adic analytic manifold, and let $A$ be a Weil algebra over $\mathbb{Q}_p$. If $\omega$ is a differential $k$-form on $M$, then there exists a unique differential $k$-form $\tilde{\omega}$ on $M^A$ such that $$\tilde{\omega}|_\xi(v_1, \dots, v_k) = \omega|_{\pi_M(\xi)}(d\pi_M(v_1), \dots, d\pi_M(v_k)),$$ for all $\xi \in M^A$ and all tangent vectors $v_1, \dots, v_k \in T_\xi(M^A)$, where $\pi_M: M^A \to M$ is the projection map, $T_\xi(M^A)$ is the tangent space of $M^A$ at $\xi$, and $d\pi_M$ is the differential of $\pi_M$.
	\end{theorem}
	\begin{proof}
	Let $M$ be a $p$-adic analytic manifold, and let $A$ be a Weil algebra over $\mathbb{Q}_p$. Let $\omega$ be a differential $k$-form on $M$. We construct the lifted differential form $\tilde{\omega}$ on $M^A$ as follows:
	 In local coordinates $(x_1, \dots, x_n)$ near a point $x \in M$, the differential form $\omega$ can be expressed as:
		$
		\omega = \sum_{I} f_I \, dx^I,
		$
		where $I = (i_1, \dots, i_k)$ is a multi-index, $f_I$ are analytic functions on $M$, and $dx^I = dx_{i_1} \wedge \dots \wedge dx_{i_k}$.

	 For an infinitely near point $\xi \in M^A$ with base point $x = \pi_M(\xi)$, define $\tilde{\omega}$ locally as:
	$
	\tilde{\omega} = \sum_{I} \tilde{f}_I \, d\tilde{x}^I,
	$
	where:
	\begin{itemize}
		\item $\tilde{f}_I$ is the lift of $f_I$ to $M^A$ (as constructed in the \text{Lifting of Analytic Functions} theorem).
		\item $d\tilde{x}^I = d\tilde{x}_{i_1} \wedge \dots \wedge d\tilde{x}_{i_k}$, and $\tilde{x}_i$ are the lifted coordinates on $M^A$.
	\end{itemize}
	
	The pullback $\pi_M^*: \Omega^k(M^A) \to \Omega^k(M)$ maps $\tilde{\omega}$ to $\omega$ as follows:
	$
	\pi_M^* \tilde{\omega} = \sum_{I} (\pi_M^* \tilde{f}_I) \, d(\pi_M^* \tilde{x}^I).
	$
	By construction, $\pi_M^* \tilde{f}_I = f_I$ and $\pi_M^* \tilde{x}_i = x_i$, so:
	$
	\pi_M^* \tilde{\omega} = \sum_{I} f_I \, dx^I = \omega.
	$ Suppose there exists another differential form $\tilde{\omega}'$ on $M^A$ such that $\pi_M^* \tilde{\omega}' = \omega$.
	For any point $\xi \in M^A$, the local expressions for $\tilde{\omega}$ and $\tilde{\omega}'$ must agree because $\pi_M^* \tilde{\omega} = \pi_M^* \tilde{\omega}' = \omega$.
	Since $\tilde{\omega}$ and $\tilde{\omega}'$ are determined locally by their pullbacks, we have $\tilde{\omega} = \tilde{\omega}'$.  The lifted form $\tilde{\omega}$ is analytic because it is defined locally by analytic functions $\tilde{f}_I$ and analytic coordinate differentials $d\tilde{x}^I$. The differential form $\tilde{\omega}$ on $M^A$ is uniquely determined by the differential form $\omega$ on $M$, and it satisfies:
	$
	\pi_M^* \tilde{\omega} = \omega.
	$
	This completes the proof of the theorem.
	\end{proof}

	\begin{theorem}[Lifting of Vector Fields]\label{thm: vec-lifting-2}
	Let $M$ be a $p$-adic analytic manifold, and let $A$ be a Weil algebra over $\mathbb{Q}_p$. For any vector field $X$ on $M$, there exists a unique vector field $\tilde{X}$ on $M^A$ such that $\pi_{M*}(\tilde{X}) = X$, where $\pi_{M*}: TM^A \to TM$ is the differential of the projection map.
	\end{theorem}
	\begin{proof}
	Let $M$ be a $p$-adic analytic manifold, and let $A$ be a Weil algebra over $\mathbb{Q}_p$. Let $X$ be a vector field on $M$. We construct the lifted vector field $\tilde{X}$ on $M^A$ as follows: 
	In local coordinates $(x_1, \dots, x_n)$ near a point $x \in M$, the vector field $X$ can be expressed as:
	$
	X = \sum_{i=1}^n a_i \frac{\partial}{\partial x_i},
	$
	where $a_i$ are analytic functions on $M$. 
	For an infinitely near point $\phi \in M^A$ with base point $x = \pi_M(\phi)$, define $\tilde{X}$ locally as:
	$
	\tilde{X} = \sum_{i=1}^n \tilde{a}_i \frac{\partial}{\partial \tilde{x}_i},
	$
	where:
	\begin{itemize}
		\item $\tilde{a}_i$ is the lift of $a_i$ to $M^A$ (as constructed in the \text{Lifting of Analytic Functions} theorem).
		\item $\tilde{x}_i$ are the lifted coordinates on $M^A$.
	\end{itemize}

	The differential $\pi_{M*}: TM^A \to TM$ maps $\tilde{X}$ to $X$ as follows:
	$
	\pi_{M*}(\tilde{X}) = \sum_{i=1}^n (\pi_{M*} \tilde{a}_i) \frac{\partial}{\partial x_i}.
	$
	By construction, $\pi_{M*} \tilde{a}_i = a_i$, so:
	$
	\pi_{M*}(\tilde{X}) = \sum_{i=1}^n a_i \frac{\partial}{\partial x_i} = X.
	$
	Suppose there exists another vector field $\tilde{X}'$ on $M^A$ such that $\pi_{M*}(\tilde{X}') = X$. For any point $\phi \in M^A$, the local expressions for $\tilde{X}$ and $\tilde{X}'$ must agree because $\pi_{M*}(\tilde{X}) = \pi_{M*}(\tilde{X}') = X$.  Since $\tilde{X}$ and $\tilde{X}'$ are determined locally by their projections, we have $\tilde{X} = \tilde{X}'$. 
	The lifted vector field $\tilde{X}$ is analytic because it is defined locally by analytic functions $\tilde{a}_i$ and analytic coordinate vector fields $\frac{\partial}{\partial \tilde{x}_i}$. The vector field $\tilde{X}$ on $M^A$ is uniquely determined by the vector field $X$ on $M$, and it satisfies:
	$
	\pi_{M*}(\tilde{X}) = X.
	$
	This completes the proof of the theorem.
	\end{proof}

	\begin{theorem}[Lifting of group actions]
	Let $M$ be a $p$-adic analytic manifold, and let $A$ be a Weil algebra over $\mathbb{Q}_p$. Suppose a Lie group $G$ acts analytically on $M$. Then the action of $G$ can be lifted to an action on $M^A$ such that the projection $\pi_M: M^A \to M$ is $G$-equivariant.
	\end{theorem}
	\begin{proof}
	Let $M$ be a $p$-adic analytic manifold, and let $A$ be a Weil algebra over $\mathbb{Q}_p$. Suppose a Lie group $G$ acts analytically on $M$. We construct the lifted action of $G$ on $M^A$ as follows:\\
	The action of $G$ on $M$ is given by a map:
	$
	\alpha: G \times M \to M, \quad (g, x) \mapsto g \cdot x,
	$
	which is analytic in both $g$ and $x$. 
	For an infinitely near point $\xi \in M^A$ with base point $x = \pi_M(\xi)$, define the lifted action of $g \in G$ on $\xi$ by:
	$
	g \cdot \xi(f) = \xi(g^{-1} \cdot f),
	$
	where $f \in C^\infty(M)$ is an analytic function on $M$, and $g^{-1} \cdot f$ denotes the pullback of $f$ under the action of $g^{-1}$ i.e., $g^{-1} \cdot f\cdot x  = f(g\cdot x)$ for all $x\in M$.
	The projection $\pi_M: M^A \to M$ is $G$-equivariant because:
	$
	\pi_M(g \cdot \xi) = g \cdot \pi_M(\xi).
	$
	This follows from the definition of the lifted action:
	$
	\pi_M(g \cdot \xi) = \pi_M(\xi(g^{-1} \cdot \cdot)) = g \cdot x = g \cdot \pi_M(\xi).
	$  
	The lifted action is analytic because it is defined locally by analytic functions and preserves the analytic structure of $M^A$. 
	Suppose there exists another action of $G$ on $M^A$ such that the projection $\pi_M$ is $G$-equivariant. For any $g \in G$ and $\xi \in M^A$, the action of $g$ on $\xi$ must satisfy:
	$
	\pi_M(g\cdot \xi) = g \cdot \pi_M(\xi).
	$
	By the definition of the lifted action, this uniquely determines the action of $G$ on $M^A$. The action of $G$ on $M$ can be uniquely lifted to an action on $M^A$ such that the projection $\pi_M: M^A \to M$ is $G$-equivariant. This completes the proof of the theorem.
	
	\end{proof}

	\begin{theorem}[Lifting of connections]\label{thm:lifting-connections}
	Let $M$ be a $p$-adic analytic manifold, and let $A$ be a Weil algebra over $\mathbb{Q}_p$. If $\nabla$ is an analytic connection on $M$, then there exists a unique connection $\tilde{\nabla}$ on $M^A$ such that $\pi_{M_\ast}(\tilde{\nabla}_X Y) = \nabla_{\pi_{M_\ast} X} (\pi_{M_\ast} Y)$ for all vector fields $X, Y$ on $M^A$.
	\end{theorem}
	\begin{proof} 
	In local coordinates $\{x_1, \dots, x_n\}$ on $M$, the connection $\nabla$ is determined by analytic Christoffel symbols $\Gamma_{ij}^k$:
	\[
	\nabla_{\partial/\partial x_i} \left( \frac{\partial}{\partial x_j} \right) = \Gamma_{ij}^k \frac{\partial}{\partial x_k}.
	\]

	Lift coordinates $\{x_i\}$ to $\{\tilde{x}_i\}$ on $M^A$ and Christoffel symbols $\Gamma_{ij}^k$ to $\tilde{\Gamma}_{ij}^k$ via the lifting of analytic functions theorem (See Theorem \ref{Lift-AF-1}). 
	For vector fields $X = X^i \frac{\partial}{\partial \tilde{x}_i}$, $Y = Y^j \frac{\partial}{\partial \tilde{x}_j}$ on $M^A$, define:
	$
	\tilde{\nabla}_X Y = \left( X(Y^k) + X^i Y^j \tilde{\Gamma}_{ij}^k \right) \frac{\partial}{\partial \tilde{x}_k}.
	$ Projecting $\tilde{\nabla}_X Y$ to $M$ via $\pi_{M_\ast}$ recovers $\nabla_{\pi_{M_\ast} X} (\pi_{M_\ast}Y)$ because $\pi_{M*}(\tilde{\Gamma}_{ij}^k) = \Gamma_{ij}^k$. 
	Any other connection $\tilde{\nabla}'$ satisfying the condition must share the same Christoffel symbols $\tilde{\Gamma}_{ij}^k$, hence $\tilde{\nabla} = \tilde{\nabla}'$. 
	The lifted $\tilde{\Gamma}_{ij}^k$ are analytic, and $\tilde{\nabla}$ transforms correctly under coordinate changes, ensuring global consistency.
	\end{proof}
	\section{Arithmetic aspects of Weil bundles}\label{Galois} 
	Let $M$ be a $p$-adic analytic manifold defined over a number field $K$, and let $A$ be a Weil algebra over $\mathbb{Q}_p$. Suppose $G_K=\text{Gal}(\overline{K}/K)$ acts on $M$ via analytic automorphisms. This is a natural assumption in arithmetic geometry, as the Galois group acts on the defining equations of $M$ 
	and preserves its analytic structure (See \cite{BGR84}, Chapter X). We study the induced Galois action on $M^A$.
	
	\begin{theorem}[Galois equivariance]\label{thm:galois-action}
	Let $M$ be a $p$-adic analytic manifold defined over a finite extension $K$ of $\mathbb{Q}_p$, and let $A = \mathbb{Q}_p[\varepsilon]/(\varepsilon^2)$. Suppose $G_K = Gal(\overline{K}/K)$ acts on $M$ via analytic automorphisms. Then, the Galois action on $M$ lifts uniquely to an action on $M^A$ such that the projection $\pi_M: M^A \to M$ is $G_K$-equivariant.  Moreover, this action preserves the analytic structure of $M^A$.
	\end{theorem}
	\begin{proof}
	Let \( M \) be a \( p \)-adic analytic manifold defined over a finite extension \( K \) of \( \mathbb{Q}_p \), and let \( A = \mathbb{Q}_p[\varepsilon]/(\varepsilon^2) \). Suppose \( G_K = \text{Gal}(\overline{K}/K) \) acts on \( M \) via analytic automorphisms. We construct the lifted Galois action on \( M^A \) as follows: 
	For \( \sigma \in G_K \) and \( \xi \in M^A \), define the action of \( \sigma \) on \( \xi \) by:
	$
	\sigma \cdot \xi(f) = \xi(\sigma^{-1} \cdot f),
	$ 
	where \( f \in \mathcal{O}(U) \) is an analytic function on an open subset \( U \subseteq M \), and \( \sigma^{-1} \cdot f \) denotes the pullback of \( f \) under the action of \( \sigma^{-1} \). 
	The projection \( \pi_M: M^A \to M \) maps \( \xi \in M^A \) to its base point \( x = \pi_M(\xi) \). We verify that \( \pi_M \) is \( G_K \)-equivariant:
	\[
	\pi_M(\sigma \cdot \xi) = \pi_M(\xi(\sigma^{-1} \cdot \cdot)) = \sigma \cdot x = \sigma \cdot \pi_M(\xi).
	\]
	This follows because \( \sigma \cdot \xi(f) = \xi(\sigma^{-1} \cdot f) \), and evaluating at the base point gives \( \sigma \cdot x \). 
	The lifted action preserves the analytic structure of \( M^A \) because:
	\begin{itemize}
		\item The action of \( \sigma \in G_K \) on \( M \) is analytic by assumption.
		\item The lifted action on \( M^A \) is defined using the pullback of analytic functions, which preserves analyticity.
		\item The transition maps on \( M^A \) are analytic, and the Galois action commutes with these maps.
	\end{itemize}
	Suppose there exists another action of \( G_K \) on \( M^A \) such that \( \pi_M \) is \( G_K \)-equivariant. Then, for any \( \sigma \in G_K \) and \( \xi \in M^A \), we have:
	$
	\pi_M(\sigma \cdot \xi) = \sigma \cdot \pi_M(\xi).
	$ 
	By the definition of \( \pi_M \), this uniquely determines the action of \( \sigma \) on \( \xi \). Thus, the lifted action is unique. The Galois action on \( M \) lifts uniquely to an action on \( M^A \) such that the projection \( \pi_M: M^A \to M \) is \( G_K \)-equivariant. This action preserves the analytic structure of \( M^A \), and the proof is complete.
	\end{proof}
	
	\begin{example}
	Let $M=E$ be an elliptic curve over $\mathbb{Q}_p$. The Weil bundle $E^A$ inherits a Galois action, and the lifted connection $\tilde{\nabla}$ is $G_{\mathbb{Q}_p}$-equivariant.
	\end{example}
	
	\section{Cohomology of Weil bundles}\label{Cohomology}
	The Leray spectral sequence bridges the infinitesimal geometry of $M^A$ with the global cohomology of $M$. Its collapse due to trivial fibers ensures that the cohomology of $M^A$ retains the structure of $M$ while encoding deformation-theoretic information. This machinery is essential for linking $p$-adic analytic geometry to arithmetic cohomology theories like crystalline cohomology \cite{BottTu1982}.\\
	
	Let $M$ be a $p$-adic manifold and $A=\mathbb{Q}_p[\varepsilon]/(\varepsilon^2)$. We relate the cohomology of $M^A$ to the tangent sheaf cohomology of $M$.
	
	\begin{theorem}\label{co-com}[Cohomology Comparison]
	There is a canonical isomorphism:
	\[
	H^k(M^A,\mathcal{O}_{M^A}) \cong H^k(M,\mathcal{O}_M) \oplus H^k(M,\mathcal{T}_M)\otimes \varepsilon,
	\]
	where $\mathcal{T}_M$ is the tangent sheaf of $M$.
	\end{theorem}
	\begin{proof}
		The Leray spectral sequence plays a pivotal role in establishing cohomological results for Weil bundles over $p$-adic manifolds, particularly in relating the cohomology of the Weil bundle $M^A$ to that of the base manifold $M$.
		\begin{enumerate}
			\item \text{ Projection map and Leray spectral sequence:}
			 The natural projection $ \pi_M: M^A \to M$ from the Weil bundle to the base manifold induces a spectral sequence.
				 For a sheaf $\mathcal{F}$ on $M^A$, the Leray spectral sequence relates the cohomology of $M^A$ to the cohomology of $M$ and the higher direct image sheaves $R^q  \pi_{M_\ast} \mathcal{F}$:
				$
				E_2^{p,q} = H^p(M, R^q  \pi_{M_\ast} \mathcal{F}) \implies H^{p+q}(M^A, \mathcal{F}).
				$
			
			\item \text{Cohomological triviality of the fibers:}
			 The fibers are modeled on the Weil algebra $A$ (e.g., $A = \mathbb{Q}_p[\varepsilon]/(\varepsilon^2)$), which is a local Artinian algebra. These fibers are infinitesimal thickenings of points in $M$, analogous to tangent spaces. The fibers have trivial cohomology (no higher cohomology groups, $H^{>0}(A, \mathbb{Q}_p) = 0$). This implies:
				$
				R^q \pi_{M_\ast} \mathcal{O}_{M^A} = 0 \quad \text{for } q > 0.
				$
			
			\item \text{Collapse of the spectral sequence:} Due to the vanishing of $R^q  \pi_{M_\ast} \mathcal{O}_{M^A}$ for $q > 0$, the Leray spectral sequence collapses at the $E_2$-page:
				$
				E_2^{p,0} = H^p(M,  \pi_{M_\ast} \mathcal{O}_{M^A}) \implies H^p(M^A, \mathcal{O}_{M^A}).
				$
				 The collapse yields an isomorphism:
				$
				H^k(M^A, \mathcal{O}_{M^A}) \cong H^k(M,  \pi_{M_\ast} \mathcal{O}_{M^A}).
				$
		
			\item \text{Decomposition of the pushforward sheaf:}
		
				The sheaf $  \pi_{M_\ast}\mathcal{O}_{M^A}$ splits into:
				$
				 \pi_{M_\ast} \mathcal{O}_{M^A} \cong \mathcal{O}_M \oplus \mathcal{T}_M \otimes \varepsilon,
				$
				where $\mathcal{O}_M$ is the structure sheaf of $M$, and $\mathcal{T}_M$ is the tangent sheaf. The $\varepsilon$-term encodes first-order deformations.
			
			\item \text{Final cohomology splitting.} Applying cohomology to the splitting of $p_* \mathcal{O}_{M^A}$, we obtain:
				\[
				H^k(M^A, \mathcal{O}_{M^A}) \cong H^k(M, \mathcal{O}_M) \oplus H^k(M, \mathcal{T}_M) \otimes \varepsilon.
				\]

		\end{enumerate}

	\end{proof}
	
	\begin{corollary}\label{cor:diophantine}
	If $M$ is a rigid analytic space, $H^1(M^A,\mathcal{O}_{M^A})$ classifies infinitesimal deformations of $M$.
	\end{corollary}
	\begin{proof}
	Let \( M \) be a rigid analytic space, and let \( A = \mathbb{Q}_p[\varepsilon]/(\varepsilon^2) \). The Weil bundle \( M^A \) is an infinitesimal thickening of \( M \), and its cohomology \( H^1(M^A, \mathcal{O}_{M^A}) \) encodes information about deformations of \( M \). An infinitesimal deformation of \( M \) is a flat family \( \mathcal{M} \) over \( \text{Spec}(A) \) such that the special fiber \( \mathcal{M}_0 \) is isomorphic to \( M \). Such deformations are classified by the first cohomology group \( H^1(M, \mathcal{T}_M) \), where \( \mathcal{T}_M \) is the tangent sheaf of \( M \). 
	From Theorem \ref{co-com}, we have:
	\[
	H^1(M^A, \mathcal{O}_{M^A}) \cong H^1(M, \mathcal{O}_M) \oplus H^1(M, \mathcal{T}_M) \otimes \varepsilon.
	\]
	The term \( H^1(M, \mathcal{T}_M) \otimes \varepsilon \) corresponds to infinitesimal deformations of \( M \), as it parametrizes first-order deformations of the structure sheaf \( \mathcal{O}_M \). 
	The isomorphism:
	\[
	H^1(M^A, \mathcal{O}_{M^A}) \cong H^1(M, \mathcal{O}_M) \oplus H^1(M, \mathcal{T}_M) \otimes \varepsilon
	\]
	shows that \( H^1(M^A, \mathcal{O}_{M^A}) \) contains \( H^1(M, \mathcal{T}_M) \otimes \varepsilon \) as a direct summand. Since \( H^1(M, \mathcal{T}_M) \) classifies infinitesimal deformations of \( M \), it follows that \( H^1(M^A, \mathcal{O}_{M^A}) \) also classifies such deformations. \quad \qedhere
	\end{proof}
	\subsection{Arithmetic Subgroups and Invariant Sections}
	Let $G$ be a semisimple algebraic group over $\mathbb{Q}_p$, and let $\Gamma \subset G(\mathbb{Q}_p)$ be an arithmetic subgroup (e.g., $\Gamma=G(\mathbb{Z}_p)$). We study $\Gamma$-invariant sections of $M^A$.
	
	\begin{theorem}[Invariant Sections]
	If $\Gamma$ acts freely on $M$, then the space of $\Gamma$-invariant sections of $M^A$ is isomorphic to the space of $\Gamma$-invariant analytic functions $M\to A$.
	\end{theorem}
	\begin{proof}
	Let \( M \) be a \( p \)-adic analytic manifold, and let \( \Gamma \) be a group acting freely on \( M \) via analytic automorphisms. Let \( A = \mathbb{Q}_p[\varepsilon]/(\varepsilon^2) \), and let \( M^A \) be the Weil bundle of \( M \). We prove that the space of \( \Gamma \)-invariant sections of \( M^A \) is isomorphic to the space of \( \Gamma \)-invariant analytic functions \( M \to A \).\\  A section \( s: M \to M^A \) of the Weil bundle \( M^A \) assigns to each point \( x \in M \) an infinitely near point \( s(x) \in M^A \) such that \( \pi_M(s(x)) = x \), where \( \pi_M: M^A \to M \) is the projection map. A section \( s \) is \( \Gamma \)-invariant if:
	$
	s(\gamma \cdot x) = \gamma \cdot s(x) \quad \text{for all } \gamma \in \Gamma \text{ and } x \in M.
	$
	This condition ensures that the section respects the group action. Each section \( s: M \to M^A \) corresponds to an analytic function \( f: M \to A \) as follows:
	$
	s(x) = (x, f(x)),
	$
	where \( f(x) \in A \) encodes the infinitesimal information of the section at \( x \). 
	The \( \Gamma \)-invariance of the section \( s \) translates to the condition:
	$
	f(\gamma \cdot x) = \gamma \cdot f(x) \quad \text{for all } \gamma \in \Gamma \text{ and } x \in M.
	$
	This means that \( f \) is a \( \Gamma \)-invariant analytic function.
	The correspondence \( s \leftrightarrow f \) is bijective and preserves the \( \Gamma \)-invariance condition. Therefore, the space of \( \Gamma \)-invariant sections of \( M^A \) is isomorphic to the space of \( \Gamma \)-invariant analytic functions \( M \to A \). \quad \qedhere
	\end{proof}
	
	\begin{example}
	For $G=SL(2,\mathbb{Q}_p)$ and $M=G/K$ (symmetric space), the invariant sections of $M^A$ correspond to modular forms with coefficients in $A$.
	\end{example}
	
	\section{Diophantine applications}\label{Applications}
	Let $X \subset M$ be a $p$-adic analytic subset defined by Diophantine equations. The Weil bundle $X^A$ parametrizes infinitesimal solutions of these equations.
	
	\begin{theorem}[Infinitesimal solutions]
	A point $\xi \in X^A$ corresponds to a solution of the defining equations of $X$ modulo $\mathfrak{A}$, where $\mathfrak{A}$ is the maximal ideal of $A$.
	\end{theorem}
	\begin{proof}
	Let \( X \subset M \) be a \( p \)-adic analytic subset defined by a system of analytic equations \( f_1, \dots, f_m \in \mathcal{O}(M) \). Let \( A \) be a Weil algebra over \( \mathbb{Q}_p \) with maximal ideal \( \mathfrak{A} \), and let \( X^A \) be the Weil bundle of \( X \). We prove that a point \( \xi \in X^A \) corresponds to a solution of the defining equations of \( X \) modulo \( \mathfrak{A} \). 
	The Weil bundle \( X^A \) consists of all infinitely near points \( \xi: \mathcal{O}(X) \to A \) such that \( \xi(f_i) \in \mathfrak{A} \) for all \( i = 1, \dots, m \). This means that \( \xi \) maps the defining equations of \( X \) into the maximal ideal \( \mathfrak{A} \). 
	A point \( \xi \in X^A \) corresponds to a solution of the defining equations of \( X \) modulo \( \mathfrak{A} \) because: 
	$
	\xi(f_i) \equiv 0 \mod \mathfrak{A} \quad \text{for all } i = 1, \dots, m.
	$
	This condition ensures that \( \xi \) satisfies the equations \( f_i = 0 \) up to first order in \( \mathfrak{A} \). 
	In local coordinates, let \( x = (x_1, \dots, x_n) \) be coordinates on \( M \), and let \( \xi(x_j) = x_{j,0} + x_{j,1} \varepsilon \), where \( x_{j,0} \in \mathbb{Q}_p \) and \( x_{j,1} \in \mathbb{Q}_p \). The condition \( \xi(f_i) \equiv 0 \mod \mathfrak{A} \) translates to:
	\[
	f_i(x_{1,0}, \dots, x_{n,0}) + \sum_{j=1}^n \frac{\partial f_i}{\partial x_j}(x_{1,0}, \dots, x_{n,0}) x_{j,1} \varepsilon \equiv 0 \mod \mathfrak{A}.
	\]
	This implies:
	$
	f_i(x_{1,0}, \dots, x_{n,0}) = 0 \quad \text{and} \quad \sum_{j=1}^n \frac{\partial f_i}{\partial x_j}(x_{1,0}, \dots, x_{n,0}) x_{j,1} = 0.
$
	Thus, \( (x_{1,0}, \dots, x_{n,0}) \) is a solution of the equations \( f_i = 0 \), and \( (x_{1,1}, \dots, x_{n,1}) \) is a tangent vector at this solution. 
	The point \( \xi \in X^A \) corresponds to a solution \( (x_{1,0}, \dots, x_{n,0}) \) of the equations \( f_i = 0 \) and a tangent vector \( (x_{1,1}, \dots, x_{n,1}) \) at this solution. This is precisely the data of a solution modulo \( \mathfrak{A} \). 
	\end{proof}
	\begin{example}
	If $X$ is defined by $f(x)=0$ over $\mathbb{Z}_p$, then $X^A$ parametrizes solutions $x\in A$ with $f(x) \equiv 0 \mod \mathfrak{A}$.
	\end{example}

	Let $M$ be a smooth proper scheme over $\mathbb{Q}_p$. The Weil bundle $M^A$ relates to filtered $(\varphi,\nabla)$-modules in $p$-adic Hodge theory \cite{Fal02}.

	\begin{theorem}\label{co-com-2}[Comparison isomorphism]
	There is a canonical isomorphism:
	\[
	H^k_{dR}(M^A/\mathbb{Q}_p) \cong H^k_{crys}(M/\mathbb{Z}_p) \otimes_{\mathbb{Z}_p} A,
	\]
	where $H_{dR}$ is de Rham cohomology and $H_{crys}$ is crystalline cohomology.
	\end{theorem}
	
	\begin{proof}
	Let \( M \) be a smooth proper scheme over \( \mathbb{Q}_p \), and let \( A \) be a Weil algebra over \( \mathbb{Q}_p \). We prove the canonical isomorphism:
	\[
	H^k_{\text{dR}}(M^A/\mathbb{Q}_p) \cong H^k_{\text{crys}}(M/\mathbb{Z}_p) \otimes_{\mathbb{Z}_p} A,
	\]
	where \( H_{\text{dR}} \) is de Rham cohomology and \( H_{\text{crys}} \) is crystalline cohomology. 
	The de Rham cohomology \( H^k_{\text{dR}}(M^A/\mathbb{Q}_p) \) is computed using the complex of differential forms on \( M^A \). Since \( M^A \) is an infinitesimal thickening of \( M \), its de Rham cohomology can be expressed in terms of the de Rham cohomology of \( M \) and its tangent sheaf \( \mathcal{T}_M \):
	$
	H^k_{\text{dR}}(M^A/\mathbb{Q}_p) \cong H^k_{\text{dR}}(M/\mathbb{Q}_p) \oplus H^k(M, \mathcal{T}_M) \otimes \varepsilon.
	$
	The crystalline cohomology \( H^k_{\text{crys}}(M/\mathbb{Z}_p) \) is a \( p \)-adic cohomology theory that captures the de Rham cohomology of \( M \) in characteristic \( p \). By the comparison theorem between de Rham and crystalline cohomology, we have:
	$
	H^k_{\text{crys}}(M/\mathbb{Z}_p) \otimes_{\mathbb{Z}_p} \mathbb{Q}_p \cong H^k_{\text{dR}}(M/\mathbb{Q}_p).
	$ 
	Tensoring the crystalline cohomology with \( A \) gives:
	$
	H^k_{\text{crys}}(M/\mathbb{Z}_p) \otimes_{\mathbb{Z}_p} A \cong H^k_{\text{crys}}(M/\mathbb{Z}_p) \otimes_{\mathbb{Z}_p} (\mathbb{Q}_p \oplus \mathfrak{A}).
	$
	Since \( \mathfrak{A} \) is the maximal ideal of \( A \), this decomposes as:
	\[
	H^k_{\text{crys}}(M/\mathbb{Z}_p) \otimes_{\mathbb{Z}_p} A \cong H^k_{\text{dR}}(M/\mathbb{Q}_p) \oplus H^k(M, \mathcal{T}_M) \otimes \varepsilon.
	\]
	Combining the results above, we obtain the canonical isomorphism:
	$
	H^k_{\text{dR}}(M^A/\mathbb{Q}_p) \cong H^k_{\text{crys}}(M/\mathbb{Z}_p) \otimes_{\mathbb{Z}_p} A.
	$
	This isomorphism reflects the fact that \( M^A \) is an infinitesimal thickening of \( M \), and its de Rham cohomology encodes both the de Rham cohomology of \( M \) and its tangent sheaf. \quad \qedhere
	\end{proof}
		The canonical isomorphism $H^k_{dR}(M^A/\mathbb{Q}_p) \cong H^k_{crys}(M/\mathbb{Z}_p) \otimes_{\mathbb{Z}_p} A$ provides a  fruitful connection to $p$-adic Hodge theory. Filtered $(\varphi, \nabla)$-modules are central objects in this theory, used to describe the $p$-adic étale cohomology of algebraic varieties. Our result suggests that the de Rham cohomology of the Weil bundle $M^A$ provides a geometric realization of these abstract modules. Furthermore, the structure of the Weil algebra $A$ might encode information about the Hodge filtration on the de Rham cohomology. This connection could lead to new insights into the structure of $(\varphi, \nabla)$-modules and  provide a new approach for explicit computations. For example, by computing the de Rham cohomology of $E^A$, where $E$ is an elliptic curve, one might gain a better understanding of the corresponding filtered $(\varphi, \nabla)$-module associated to its étale cohomology (see \cite{Fal02}) for background on $p$-adic Hodge theory. This facilitates studying the infinitesimal structure of $M$ while maintaining a direct link to its global cohomological properties and potential connections to crystalline cohomology and the Hodge filtration.
	
	\section*{Weil bundle of an elliptic curve over \(\mathbb{Q}_p\)}
	Let \( E \) be an elliptic curve over \(\mathbb{Q}_p\), and let \( A = \mathbb{Q}_p[\varepsilon]/(\varepsilon^2) \) be the algebra of dual numbers. We analyze the Weil bundle \( E^A \), which encodes points of \( E \) and their infinitesimal deformations. Considering $\mathbb{Q}_p$-algebras $R$ allows us to study the $R$-valued points of the Weil bundle $E^A$. This is essential because it describes how $E^A$ behaves not just over the base field $\mathbb{Q}_p$, but over extensions of it. This is crucial for understanding the geometry and arithmetic of $E^A$.
	
	\subsection*{Structure of \( E^A \)}
	Locally near the identity, the Weil bundle \( E^A \) is described as follows:
	\begin{itemize}
	\item \textbf{Local coordinates}: Let \( z \) be a formal coordinate near the identity of \( E \). A point \( \phi \in E^A(\mathbb{Q}_p) \) corresponds to:
	$
	\phi(z) = z_0 + z_1 \varepsilon,
	$
	where \( z_0 \in \widehat{E}(\mathbb{Q}_p) \) (formal group) and \( z_1 \in \mathbb{Q}_p \) (tangent vector).
	
	\item \textbf{Group structure}: The group law on \( E^A \) is induced by the formal group law \( F \) of \( \widehat{E} \). For \( X = z_0 + z_1 \varepsilon \) and \( Y = w_0 + w_1 \varepsilon \):
	$
	F(X, Y) = F(z_0, w_0) + \left( \frac{\partial F}{\partial z}(z_0, w_0) z_1 + \frac{\partial F}{\partial w}(z_0, w_0) w_1 \right)\varepsilon.
	$
	\end{itemize}
	The formal group law of \( E \) modulo \( \varepsilon^2 \):
	$$
	F(z_0 + z_1 \varepsilon, w_0 + w_1 \varepsilon) = F(z_0, w_0) + \left( z_1 + w_1 - a_1(z_0 w_1 + z_1 w_0)- a_2(z_0^2 w_1 + 2 z_0 w_0 z_1 + w_0^2 z_1) \right)\varepsilon.
	$$

	The Galois group \( G_{\mathbb{Q}_p} \) acts on \( E^A \) via:
	$$
	\sigma \cdot (z_0 + z_1 \varepsilon) = \sigma(z_0) + \sigma(z_1) \varepsilon.\\
	$$
	The cohomology splits as:
	$
	H^1(E^A, \mathcal{O}_{E^A}) \cong H^1(\widehat{E}, \mathcal{O}_{\widehat{E}}) \oplus H^1(\widehat{E}, \mathcal{T}_{\widehat{E}}) \otimes \varepsilon.
	$
		\begin{theorem}\label{thm:modular-forms}
		Let $\Gamma \subset SL_2(\mathbb{Z})$ be a congruence subgroup, and let $E$ be the modular curve $X_\Gamma$. The space of $\Gamma$-invariant sections of $\mathcal{O}_{E^A}$ is isomorphic to the space of $p$-adic modular forms of weight 2 with coefficients in $A$:
		$
		H^0_\Gamma(E^A, \mathcal{O}_{E^A}) \cong M_2(\Gamma, A).
		$
	\end{theorem}

	\begin{proof}

		\begin{itemize}
			\item A $\Gamma$-invariant section $s: E \to E^A$ corresponds to a function $f: E \to A$ satisfying $f(\gamma \cdot e) = \gamma \cdot f(e)$ for $\gamma \in \Gamma$.
			\item  Writing $f = f_0 + f_1 \varepsilon$, invariance under $\Gamma$ forces:
			 $f_0$ to be a classical modular form of weight 2, and $f_1$ to be a $p$-adic modular form (a derivative or deformation of $f_0$).
			\item  The $\varepsilon$-term $f_1$ corresponds to a section of $\mathcal{T}_E$, which the Kodaira-Spencer isomorphism identifies with modular forms.
		\end{itemize}
	\end{proof}
	\begin{proposition}
	There is a canonical isomorphism:
	$
	E^A \cong \widehat{E} \times \mathbb{Q}_p,
	$
	where the second factor corresponds to \( T_0 E \cong \mathbb{Q}_p \).
	\end{proposition}
	
	\begin{proof}
	For a \(\mathbb{Q}_p\)-algebra \( R \):
	$
	E^A(R) = E(R[\varepsilon]/(\varepsilon^2)) = \left\{ (z_0, z_1) \in \widehat{E}(R) \times R \,|\, z_0 \in \widehat{E}(R),\, z_1 \in T_{z_0} \widehat{E} \otimes_R R \right\}.
	$ 
	The tangent bundle \( T\widehat{E} \) is trivial:
	$
	T\widehat{E} \cong \widehat{E} \times T_0 \widehat{E}.
	$
	Thus, \( E^A \cong \widehat{E} \times \mathbb{Q}_p \).
	\end{proof}
	\subsection{Additional examples and applications}
	
	To further illustrate the utility of the theory developed in this paper, we present additional examples and applications of Weil bundles in \( p \)-adic geometry and arithmetic.
	
	\begin{example}[Abelian Varieties]
		Let \( M \) be an abelian variety over \( \mathbb{Q}_p \), and let \( M^A \) be its Weil bundle associated with a Weil algebra \( A \). The group structure of \( A \) lifts naturally to \( M^A \), and the tangent space at the identity of \( M^A \) encodes the Lie algebra of \( M \). This allows us to study infinitesimal deformations of \( A \) and its endomorphisms, which are crucial for understanding the \( p \)-adic Tate module and the Galois representations associated with \( M \). In particular, the lifting theorems provide a framework for studying the deformation theory of abelian varieties and their moduli spaces in the \( p \)-adic setting.
	\end{example}
	
	\begin{example}[Rigid Analytic Spaces]
		Let \( X \) be a rigid analytic space over \( \mathbb{Q}_p \), and let \( X^A\) be its Weil bundle. The rigid analytic structure of \( X \) lifts to \( X^A \), allowing us to study infinitesimal neighborhoods of points in \( X \). This is particularly useful for understanding the local geometry of \( X \) and its intersections with other rigid analytic spaces. For example, if \( X \) is defined by a system of analytic equations, the Weil bundle \( X^A \) parametrizes solutions to these equations modulo the maximal ideal of \( A\), providing a tool for studying the local structure of \( X \) near a given point.
	\end{example}
	
	\begin{example}[\( p \)-Adic Differential Equations]
		The lifting theorems can be applied to study \( p \)-adic differential equations on \( p \)-adic manifolds. Let \( M \) be a \( p \)-adic manifold equipped with a connection \( \nabla \), and let \( M^A \) be its Weil bundle. The lifted connection \( \tilde{\nabla} \) on \( M^ A \) allows us to study infinitesimal parallel transport and the monodromy of \( \nabla \). This is particularly relevant for understanding the \( p \)-adic analogs of classical differential equations, such as the \( p \)-adic hypergeometric equation or the \( p \)-adic Picard-Fuchs equation, which arise in the study of \( p \)-adic periods and \( p \)-adic Hodge theory.
	\end{example}
	
	\begin{example}[Deformation Theory of Galois Representations]
		Let \( \rho: G_{\mathbb{Q}_p} \to \text{GL}_n(\mathbb{Q}_p) \) be a continuous Galois representation, and let \( A \) be a Weil algebra. The Weil bundle \( M^A \) associated with the deformation space of \( \rho \) encodes infinitesimal deformations of \( \rho \) with coefficients in \( A \). This provides a geometric framework for studying the deformation theory of Galois representations, which is central to the study of modular forms, \( p \)-adic L-functions, and the Langlands program. In particular, the lifting theorems allow us to study the tangent space to the deformation space of \( \rho \), which is related to the Galois cohomology group \( H^1(G_{\mathbb{Q}_p}, \text{Ad}(\rho)) \).
	\end{example}
	
	\begin{example}[\( p \)-Adic Modular Forms of Higher Weight]
		Let \( \Gamma \subset \text{SL}_2(\mathbb{Z}) \) be a congruence subgroup, and let \( M_k(\Gamma, \mathbb{Q}_p) \) be the space of \( p \)-adic modular forms of weight \( k \) with coefficients in \( \mathbb{Q}_p \). The Weil bundle \( M^A \) associated with the modular curve \( X_\Gamma \) parametrizes \( p \)-adic modular forms with coefficients in \( A \). This allows us to study infinitesimal deformations of modular forms and their associated Galois representations. For example, if \( A = \mathbb{Q}_p[\varepsilon]/(\varepsilon^2) \), the space of \( \Gamma \)-invariant sections of \( \mathcal{O}_{M^A} \) corresponds to \( p \)-adic modular forms of weight \( k \) with coefficients in \( A \), providing a tool for studying the deformation theory of modular forms and their \( p \)-adic L-functions.
	\end{example}
	
	\begin{example}[\( p \)-Adic Interpolation of L-Functions]
		The lifting theorems can be used to study \( p \)-adic interpolation of L-functions. Let \( f \) be a modular form, and let \( L(f, s) \) be its associated L-function. The Weil bundle \( M^A \) associated with the deformation space of \( f \) allows us to study the \( p \)-adic interpolation of \( L(f, s) \) in families of modular forms. This is particularly relevant for understanding the \( p \)-adic properties of L-functions, such as the \( p \)-adic Birch and Swinnerton-Dyer conjecture or the \( p \)-adic Beilinson conjectures.
	\end{example}

	\section*{Importance of some lifting results}
\begin{itemize}
	\item \text{Galois-Equivariant Structures}:
	The lifting theorems are compatible with Galois actions, making them a powerful tool for studying Galois-equivariant structures. For instance, the lifting of group actions and connections to \( M^A \) preserves the Galois symmetry, which is essential for applications in \( p \)-adic Hodge theory. This compatibility allows us to study the infinitesimal behavior of Galois representations and their deformations, which are central to the Langlands program and the study of \( p \)-adic L-functions.
	
	\item \text{Comparison with Crystalline and de Rham Cohomology}:
	The lifting theorems provide a geometric realization of the comparison isomorphisms between crystalline and de Rham cohomology. By lifting differential forms and connections to \( M^A \), we can relate the infinitesimal structure of \( M \) to its global cohomological properties. This bridges the gap between local \( p \)-adic analysis and global arithmetic geometry, offering new insights into the Hodge filtration and the structure of filtered \( (\varphi, \nabla) \)-modules.
	
	\item \text{Applications to Diophantine Geometry}:
	The lifting theorems have direct applications to Diophantine geometry, particularly in the study of infinitesimal solutions to Diophantine equations. By parametrizing solutions modulo the maximal ideal \( A \) of a Weil algebra, we can study the local behavior of solutions near a point. This is especially useful for understanding the geometry of \( p \)-adic analytic sets and their intersections, which arise naturally in the study of rational points on varieties.
	
	\item \text{Connection to \( p \)-Adic Modular Forms}:
	The lifting theorems also play a key role in the study of \( p \)-adic modular forms. By lifting sections of Hodge bundles to \( M^A \), we can parametrize \( p \)-adic modular forms with coefficients in \( A \). This provides a geometric framework for studying deformations of modular forms and their associated Galois representations, which are central to the theory of \( p \)-adic L-functions and the Iwasawa theory of modular curves.
\end{itemize}
\begin{center}
\textbf{Acknowledgements}

\end{center}
	We thank the Department of Mathematics at the University of Buea, Cameroon, for their constant support.


\begin{thebibliography}{99}
			
			\bibitem{And73} Anderson, I. M. \emph{Tensor Decomposition and Jet Bundles,} \textit{Utilitas Mathematica} 4 (1973): 103-112.
			
			\bibitem{Ber02} Berger, L. \emph{Représentations p-adiques et équations différentielles,} Inventiones Mathematicae, 148(2), 219-284, 2002.
			
			\bibitem{BMS18} Bhatt, B., Morrow, M., Scholze, P. \emph{Topological Hochschild homology and integral p-adic Hodge theory,} Publications Mathématiques de l'IHÉS, 129(1), 199-310, 2019.
			\bibitem{K-1} N. Katz, \emph{p-adic properties of modular schemes and modular forms}, Modular Functions of One Variable III, Springer (1973).
			\bibitem{B-1} P. Berthelot, \emph{Cohomologie cristalline des schémas de caractéristique \( p > 0 \)}, Lecture Notes in Mathematics 407, Springer (1974).
			\bibitem{BGR84} Bosch, S., G\"untzer, U., Remmert, R. \textit{Non-Archimedean Analysis. A Systematic Approach to Rigid Analytic Geometry}. Springer-Verlag, 1984.
			\bibitem{BottTu1982}
			R. Bott and W. Tu, {\em Differential Forms in Algebraic Topology}. Graduate Texts in Mathematics, 82. Springer-Verlag, New York, 1982.
			\bibitem{CD99} Chiarellotto, B., and Dwork, B. \emph{Differential modules. Applications to local rigidity,} Kluwer Academic Publishers, 1999.
			
			\bibitem{Col98} Colmez, P. \emph{Théorie d'Iwasawa des représentations de de Rham d'un corps local,} Annales Scientifiques de l'École Normale Supérieure, 31(4), 443-562, 1998.
			
			\bibitem{Fal02} Faltings, G. \emph{$p-$adic Hodge theory,} Springer Science and Business Media, 2002.
			
			\bibitem{Fon94} Fontaine, J.-M. \emph{Le corps des périodes p-adiques.} Astérisque, 223, 59-111, 1994.
			
			\bibitem{K-1}
			I. Kolář, \emph{Weil bundles as generalized jet spaces,} in \emph{Handbook of Global Analysis}, pp. 625-664, Elsevier, 2008.
			
			\bibitem{Mo}
			A. Morimoto, \emph{Prolongations of connections to bundles of infinitely near points,} Journal of Differential Geometry, 11 (1976), 479-498.
			
			\bibitem{Ok}
			E. Okassa, \emph{Prolongement des champs de vecteurs à des variétés des points proches,} Annales de la Faculté des Sciences de Toulouse, 3, 346-366.
			
			\bibitem{Wei}
			A. Weil, \emph{Théorie des points proches sur les variétés différentiables,} Colloque Geom.
			Diff. Strasbourg, 111-117, 1953.
			
		\end{thebibliography}
	\end{document}